\documentclass[12pt]{amsart}
\usepackage{amsmath, amsthm, amssymb}
\usepackage{graphicx, subfigure}

\newtheorem{thm}{Theorem}[section]
\newtheorem{cor}[thm]{Corollary}
\newtheorem{prop}[thm]{Proposition}

\theoremstyle{definition}
\newtheorem{defn}[thm]{Definition}

\theoremstyle{remark}

\DeclareMathOperator{\comp}{comp}
\DeclareMathOperator{\co}{co}
\DeclareMathOperator{\asc}{asc}

\title[A chromatic quasisymmetric function]{A quasisymmetric function generalization of the chromatic symmetric function}
\author{Brandon Humpert}
\date{\today}

\begin{document}

\begin{abstract}
The chromatic symmetric function $X_G$ of a graph $G$ was introduced by Stanley. In this paper we introduce a quasisymmetric generalization $X^k_G$ called the $k$-chromatic quasisymmetric function of $G$ and show that it is positive in the fundamental basis for the quasisymmetric functions. Following the specialization of $X_G$ to $\chi_G(\lambda)$, the chromatic polynomial, we also define a generalization $\chi^k_G(\lambda)$ and show that evaluations of this polynomial for negative values generalize a theorem of Stanley relating acyclic orientations to the chromatic polynomial.
\end{abstract}

\maketitle

\section{Introduction}

The symbol $\mathbb{P}$ will denote the positive integers.
Let $G=(V,E)$ be a finite simple graph with vertices
$V=[n]=\{1,2,\dots,n\}$.
A \emph{proper coloring} of $G$ is a function $\kappa:V\to\mathbb{P}$ such that
$\kappa(i)\neq\kappa(j)$ whenever $ij\in E$.  Stanley \cite{RPS95}
introduced
the \emph{chromatic symmetric function}
  $$X_G = X_G(x_1,x_2,\dots) = \sum_{\text{proper colorings }\kappa}
x_{\kappa(1)}\cdots x_{\kappa(n)}$$
in commuting indeterminates $x_1,x_2,\dots$.  This invariant is a
symmetric function,
because permuting the colors does not change whether or not a given
coloring is proper.
Moreover, $X_G$ generalizes the
classical chromatic polynomial $\chi_G(\lambda)$ (which can be obtained from
$X_G$ by
setting $k$ of the indeterminates to 1 and the others to 0).

This paper is about a quasisymmetric function generalization of $X_G$,
which arose in the following context.  Recall that the Hasse diagram of a
poset $P$ is the (acyclic) directed graph with an edge $x\to y$ for each
covering relation $x<y$ of $P$.  It is natural to ask which undirected
graphs $G$ are ``Hasse graphs'', i.e., admit orientations that are Hasse
diagrams of posets.  O.~Pretzel \cite{Pret85} gave the following answer
to this question.  Call a directed graph \emph{$k$-balanced} (Pretzel used
the term ``$k$-good'') if, for every cycle $C$ of the underlying
undirected graph of $D$, walking around $C$ traverses at least $k$ edges
forward and at least $k$ edges backward.  (So ``1-balanced'' is synonymous
with ``acyclic''.)  Then $G$ is a Hasse graph if and only if it has a
2-balanced orientation. Note that the condition is more restrictive than the mere absence of
triangles; as pointed out by Pretzel, the Gr\"otzsch graph (Figure \ref{grotzsch}) is
triangle-free, but is not a Hasse graph.

For every proper coloring $\kappa$ of $G$, there is an associated acyclic
orientation defined by directing every edge toward the endpoint with the
larger color.  Accordingly, define a coloring to be \emph{$k$-balanced}
iff it induces a $k$-balanced orientation in this way. We now can define
our main object of study: the \emph{$k$-balanced chromatic quasisymmetric
function}
  $$X^k_G = X^k_G(x_1,x_2,\dots) = \sum_{\text{$k$-balanced colorings
$\kappa$}} x_{\kappa(1)}\cdots x_{\kappa(n)}.$$

For all $k\geq 1$, the power series $X^k_G$ is \emph{quasisymmetric}: that
is, if $i_m<\cdots<i_m$, and $j_1<\cdots<j_k$, then for all
$a_1,\dots,a_m$, the monomials $x_{i_1}^{a_1}\cdots x_{i_m}^{a_m}$ and
$x_{j_1}^{a_1}\cdots x_{j_m}^{a_m}$ have the same coefficient in $X^k_G$.
Moreover, $X^1_G$ is Stanley's chromatic symmetric function (because
``1-balanced'' is synonymous with ``acyclic'').

We obtain the following results:

1. A natural expansion of $X^k_G$ in terms of $P$-partitions
\cite{EC2} of the posets whose Hasse diagram is an orientation of $G$, giving a proof that $X^k_G$ is nonnegative with respect to the fundamental
basis for the quasisymmetric functions (Thm \ref{Lpos}).

2. Explicit formulas for $X^2_G$ for cycles (Prop \ref{CycleCalc}), a proof that $X^k_G$ is always symmetric for cycles (Prop \ref{CycleSymm}), and complete bipartite graphs (Thm \ref{CompBipartCalc}).

3. A reciprocity relationship between $k$-balanced
colorings and $k$-balanced orientations, generalizing Stanley's classical
theorem that \linebreak evaluating the chromatic polynomial $\chi_G(\lambda)$ at $\lambda=-1$
yields the number of acyclic orientations of $G$ (Thm \ref{ChrPolyEval}).

This paper is organized as follows.
In Section 2 the necessary background material on graphs, quasisymmetric functions, and $P$-partitions is introduced. In Section 3, we introduce the invariant $X_G^k$, the $k$-chromatic quasisymmetric function, and look at several of its properties. In Section 4, the invariant $X_G^k$ is analyzed for some special classes of graphs. In Section 5, we introduce a specialization of $X_G^k$ that generalizes the chromatic polynomial and explore its properties.

I would like to thank Kurt Luoto for pointing out the use of $P$-partitions in Theorem \ref{Lpos}, Frank Sottile for advice on Proposition \ref{CycleSymm} and my advisor, Jeremy Martin, for his immense assistance with crafting my first paper.

\section{Background}

In this section we remind the reader of definitions and facts about graphs, posets, and quasisymmetric functions which will appear in the remainder of the paper.

\subsection{Graphs and colorings}

We will assume a familiarity with standard facts and terminology from graph theory, as in \cite{BollGraphThy}. In this paper, we are primarily concerned with simple graphs whose vertex set is $[n] = \{1, 2, \ldots, n\}$.

Recall that an \textit{orientation} of a graph $G$ is a directed graph $\mathcal{O}$ with the same vertices, so that for every edge $\{i, j\}$ of $G$, exactly one of $(i, j)$ and $(j, i)$ is an edge of $\mathcal{O}$. An orientation is often regarded as giving a direction to each edge of an undirected graph.

We define a \textit{weak cycle of an orientation} to be the edges and vertices inherited from a cycle of the underlying undirected graph.

A \textit{coloring} of a graph $G$ is a map $\kappa : [n] \to \{1, 2, \ldots\}$ such that if $\kappa(i) = \kappa(j)$, then $\{i,j\}$ is not an edge of $G$. The \textit{chromatic polynomial} of $G$ is the function $\chi: \mathbb{N} \to \mathbb{N}$ where $\chi(n)$ equals the number of colorings of $G$ using the colors $\{1, 2, \ldots, n\}$. It's a well-known result that $\chi$ is a polynomial with integer coefficients. (See \cite[\S V.1]{BollGraphThy}).

\subsection{Compositions and quasisymmetric functions}

As in \cite[\S 1.2]{EC2}, a \textit{composition} $\alpha$ is an ordered list $(\alpha_1, \alpha_2, \dots, \alpha_\ell)$. The \textit{weight} of a composition is $|\alpha| = \sum \alpha_i$. If $|\alpha| = n$, we will say that $\alpha$ is a composition of $n$ and write $\alpha \models n$. The number $\ell$ is the \textit{length} of $\alpha$.

There is a bijection between compositions of $n$ and subsets of $[n-1]$ which we will use, found in \cite[\S 7.19]{EC2}. For $\alpha = (\alpha_1, \alpha_2, \ldots, \alpha_\ell)$, define $S_\alpha = \{\alpha_1, \alpha_1 + \alpha_2, \ldots, |\alpha| - \alpha_\ell\}$. For $S = \{s_1 < s_2 < \ldots < s_m\}$, define $\co(S) = (s_1, s_2 - s_1, \ldots, s_m - s_{m-1})$. It is easy to check that $\co(S_\alpha) = \alpha$ and $S_{\co(S)} = S$.

The compositions of $n$ are ordered by refinement: for $\alpha, \beta \models n$, $\alpha \prec \beta$ if and only if $S_\alpha \subsetneq S_\beta$. Notice that under the bijection above, this relation is set containment, so that this poset is isomorphic to the boolean poset of subsets of $[n-1]$.

For a permutation $\pi \in \mathfrak{S}_n$, \textit{the ascent set of $\pi$} is $$\asc(\pi) = \{i \in [n] : \pi(i) < \pi(i+1)\}.$$ We can then define \textit{the composition associated to $\pi$} as $$\co(\pi) = \co(\asc(\pi)),$$ where $\co(\pi) \models n$. The parts of $\co(\pi)$ are thus the lengths of the maximal contiguous decreasing subsequences. For example, $\co(52164783) = (3, 2, 1, 2)$.

If $p$ is a polynomial or formal power series and $m$ is a monomial, then let $[m]p$ denote the coefficient of $m$ in $p$. As in \cite[7.19]{EC2}, a \textit{quasisymmetric function} is an element $F \in \mathbb{Q}[[x_1, x_2, \ldots]]$ with the property that 
$[x_{i_1}^{a_1}x_{i_2}^{a_2},\ldots,x_{i_\ell}^{a_\ell}]F = [x_{j_1}^{a_1}x_{j_2}^{a_2},\ldots,x_{j_\ell}^{a_\ell}]F$ whenever $i_1 < i_2 < \cdots < i_\ell$ and $j_1 < j_2 < \cdots < j_\ell$. The subring of $\mathbb{Q}[[x_1, x_2, \ldots]]$ consisting of all quasisymmetric functions will be denoted $\mathcal{Q}$, and the vector space spanned by all quasisymmetric functions of degree $n$ will be denoted $\mathcal{Q}_n$. The \textit{standard basis} or \textit{monomial basis} for $\mathcal{Q}_n$ is indexed by compositions $\alpha = (\alpha_1, \alpha_2, \ldots, \alpha_\ell) \models n$, and is given by 
$$M_\alpha = \sum_{i_1<i_2<\cdots<i_\ell} x_{i_1}^{\alpha_1}x_{i_2}^{\alpha_2}\cdots x_{i_\ell}^{\alpha_\ell}.$$
Another basis for $\mathcal{Q}_n$ is the \textit{fundamental basis}, whose elements are
\begin{equation}
\label{L-basis}
L_\alpha = \sum_{\substack{i_1 \leq i_2 \leq \cdots \leq i_n \\ i_j < i_{j+1} \text{ if } j \in S_\alpha}} x_{i_1}x_{i_2} \ldots x_{i_n},
\end{equation}
where $\alpha \models n$.

Working with the bijection between sets and compositions, and utilizing the fact that the refinement poset is boolean, these bases are related by M\"obius inversion as: 
\begin{equation} \label{LtoM} L_\alpha = \sum_{\beta \succeq \alpha} M_\beta , \end{equation}
\begin{equation} \label{MtoL} M_\alpha = \sum_{\beta \succeq \alpha} (-1)^{\ell(\beta) - \ell(\alpha)} L_\beta . \end{equation}

\subsection{$P$-partitions and the quasisymmetric function of a poset}

We follow Stanley \cite[\S 4.5]{EC1},  \cite[\S 7.19]{EC2}, with the exception that what he calls a reverse strict $P$-partition, we call a $P$-partition.

A poset $P$ whose elements are a subset of $\mathbb{P}$ is called \textit{naturally labelled} if $i <_P j$ implies that $i<_\mathbb{N} j$. A $P$\textit{-partition} is a strict order-preserving map $\tau: P \to [n]$, where $P$ be a naturally labelled poset on $[n]$.

\begin{defn}
Let $\pi \in \mathfrak{S}_n$. Then a function $f: [n] \to \mathbb{P}$ is $\pi$\textit{-compatible} whenever
\begin{gather*}
f(\pi_1) \leq f(\pi_2) \leq \cdots \leq f(\pi_n) \\
\text{and} \\
f(\pi_i) < f(\pi_{i+1}) \text{ if } \pi_i < \pi_{i+1}.
\end{gather*}
\end{defn}

For all $f: [n] \to \mathbb{P}$, there exists a unique permutation $\pi \in \mathfrak{S}_n$ for which $f$ is $\pi$-compatible. Specifically, if $\{i_1 < i_2 < \cdots < i_k\}$ is the image of $f$, then we obtain $\pi$ by listing the elements of $f^{-1}(i_1)$ in increasing order, then the elements of $f^{-1}(i_2)$ in increasing order, and so on.

\begin{prop}[Lemma 4.5.3 in \cite{EC1}] \label{P-part Equiv}
Let P be a natural partial order on $[n]$, and let $\mathcal{L}_P \subseteq \mathfrak{S}_n$ be the set of linear extensions of $P$. Then $\tau: P \to \mathbb{P}$ is a $P$-partition if and only if $\tau$ is $\pi$-compatible for some $\pi \in \mathcal{L}_P$.
\end{prop}
\begin{proof}
Given a $P$-partition $\tau$, let $\pi$ be the unique permutation of $[n]$ so that $\tau$ is $\pi$-compatible. Now if $i <_P j$, then $\tau(i) < \tau(j)$, and since $\tau$ is $\pi$-compatible, $i$ must appear before $j$ in $\pi$. Thus $\pi$ is a linear extension of $P$.

On the the other hand, given a $\pi$-compatible function $\tau$ with $\pi$ a linear extension of $P$, if $i <_P j$, then $i$ appears before $j$ in $\pi$, and so $\tau(i) < \tau(j)$.
\end{proof}

We write $S_\pi$ for the set of all $\pi$-compatible functions, and $\mathcal{A}(P)$ for the set of all $P$-partitions. Then from Proposition (\ref{P-part Equiv}) we get the decomposition
\begin{equation} \label{P-part decomp} \mathcal{A}(P) = \bigsqcup_{\pi \in \mathcal{L}_P} S_\pi . \end{equation}

The form of the fundamental quasisymmetric basis given in equation (\ref{L-basis}) and the definition of $\pi$-compatibility implies that
\begin{equation*}
\sum_{\tau \in S_\pi} x^\tau 	= L_{\co(\pi)}(x).
\end{equation*}

Given a poset $P$, we define the \textit{quasisymmetric function of a poset} $K_P$ to be

$$ K_P(x)	 = \sum_{\tau \in \mathcal{A}(P)} x^\tau .$$
In the case that $P$ is naturally labelled, we also have from \cite[Corollary 7.19.6]{EC2} that
\begin{align} 
K_P(x)		& = \sum_{\pi \in \mathcal{L}_P} \sum_{\tau \in S_\pi} x^\tau \label{KPdef} \\
			& = \sum_{\pi \in \mathcal{L}_P} L_{\co(\pi)}(x), \nonumber
\end{align}
where the first equality here is from equation (\ref{P-part decomp}). Further, notice that for any two natural relabellings $P', P''$ of a poset $P$, we have $\mathcal{L}_{P'} = \mathcal{L}_{P''}$, and thus from equation (\ref{KPdef}), $K_{P'} = K_{P''}$. So, even though $P$ may not be naturally labelled, we can use the above to calculate $K_P$.

\section{The $k$-chromatic quasisymmetric function of a graph}

Given a poset $P$ on $[n]$, define $G_P$ to be \textit{the graph induced by $P$} with vertices $[n]$ and edges given by the covering relations of $P$. Note that $G_P$ is graph-isomorphic to the Hasse diagram of $P$.

A natural question to ask is, given an arbitrary graph, does there exist a poset which induces it? We will call any such graph a \textit{Hasse graph}.



To answer the question, we notice that a poset $P$ can be identified with an orientation $\mathcal{O}_P$ of $G_P$ which we will call the \textit{orientation induced by $P$} by directing each edge of $G_P$ towards the larger element in the covering relation. These orientations are necessarily acyclic, but they have the additional property that every weak cycle has at least 2 edges oriented both forward and backward, due to the fact that Hasse diagrams include only the covering relations of the poset. That is, weak cycles may not have all but one edge oriented consistently, as in Figure \ref{bypass}; such an obstruction is called a \textit{bypass}. Using the correspondence, we see that a graph is a Hasse graph if and only if it has such an orientation.

\begin{figure}[h]
\centering
\includegraphics[height=.1\textheight]{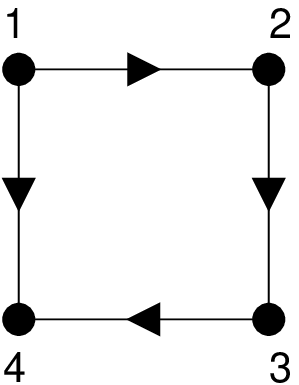}
\caption{A bypass on 4 vertices}
\label{bypass}
\end{figure}

Pretzel \cite{Pret85} observed that the above condition was related to acyclicity (in which each weak cycle has at least 1 edge oriented both forward and backward) and made the following definition.

\begin{defn}
Let $G$ be an undirected simple graph, and let $\mathcal{O}$ be an orientation of $G$. Then, for $k \geq 1$, $\mathcal{O}$ is $k$\textit{-balanced}\footnote{Pretzel used the terminology \textit{$k$-good}.} if there are at least $k$ edges oriented both forward and backward along each weak cycle. That is, given any cycle of $G$ with edges $\{v_1,v_2\},\{v_2,v_3\},\dots, \{v_{r-1},v_r\},\{v_r,v_1\}$, then $\mathcal{O}$ contains at least $k$ directed edges of the form  $(v_i,v_{i+1})$ and at least $k$ directed edges of the form $(v_{i+1},v_i)$ (where all subscripts are taken modulo~$r$).
\end{defn}

Using this definition, we can see that an orientation is acyclic if and only if it is 1-balanced. Similarly, a graph is a Hasse graph if and only if it has a 2-balanced orientation.

Recall that the \textit{girth} of a graph is the length of its smallest cycle. Then, for an orientation $\mathcal{O}$ of $G$ to be $k$-balanced, it is necessary that the girth of $G$ be at least $2k$. This is not sufficient --- the smallest counterexample is the Gr\"otzsch graph (Figure \ref{grotzsch}), which has girth 4, but does not have a 2-balanced coloring \cite{Pret85}.
(In fact, due to a result of Ne\v set\v ril and R\"odl \cite[Corollary 3]{NesRod78}, there exist graphs of arbitrarily high girth which, under any orientation, contain a bypass and are therefore not 2-balanced.)

\begin{figure}[h]
\centering
\includegraphics[height=.3\textheight]{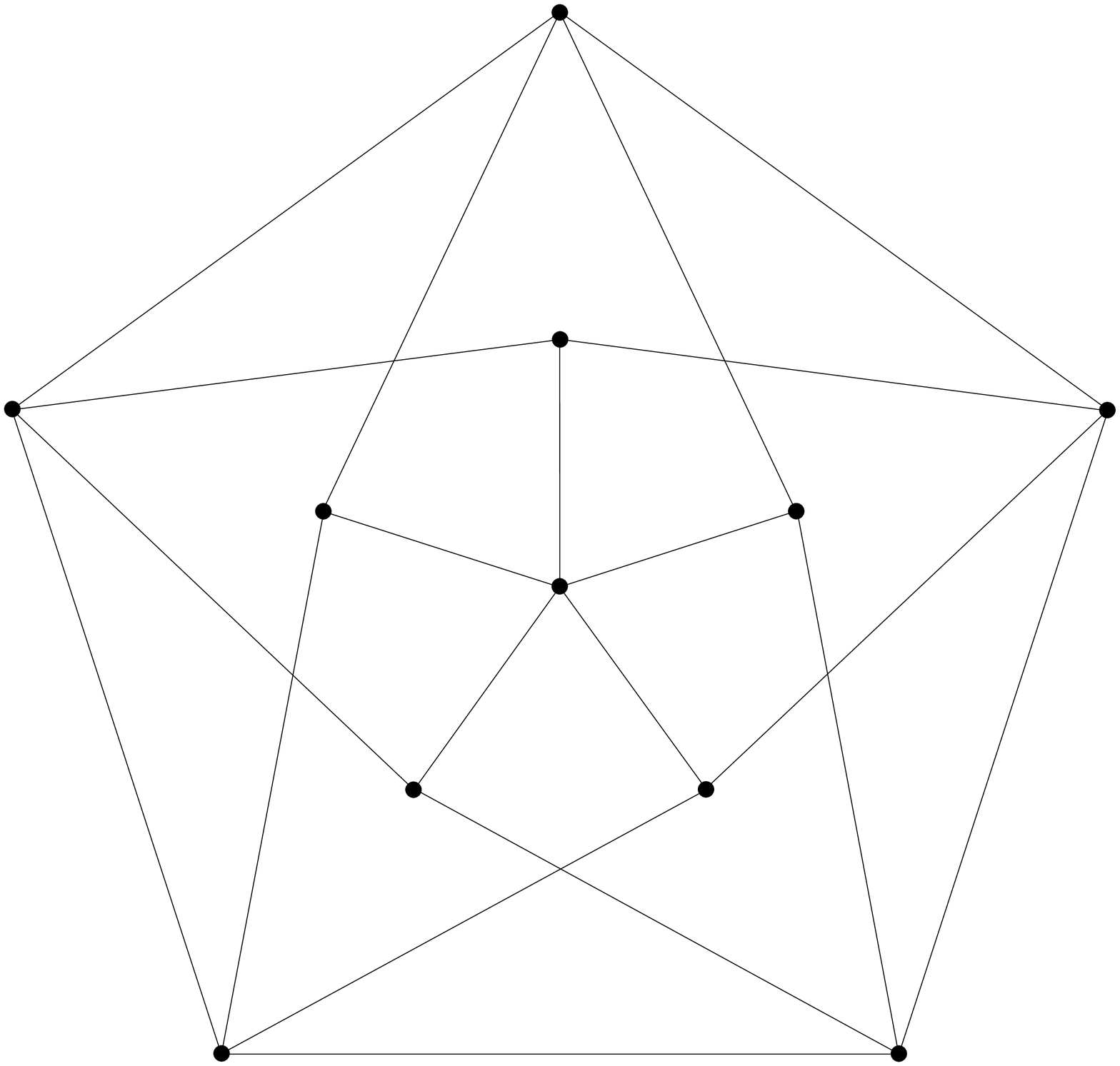}
\caption{The Gr\"otzsch graph}
\label{grotzsch}
\end{figure}

Given a poset, we have already seen that there is a corresponding acyclic orientation $\mathcal{O}_P$. Conversely, given an acyclic digraph $\mathcal{O}$, we define $P_\mathcal{O}$, \textit{the poset induced by $\mathcal{O}$}, to be the poset generated by the edges of $\mathcal{O}$. Note that the edges of $\mathcal{O}$ are the covering relations of $P_\mathcal{O}$ if and only if $\mathcal{O}$ is 2-balanced.

\begin{defn}
Let $G$ be an undirected simple graph, and let $\kappa: V(G) \to \mathbb{P}$ be a proper coloring of $G$. Then the \textit{orientation induced by $\kappa$} is the orientation $\mathcal{O}_\kappa$ where each edge is directed towards the vertex with the greater color. If $\mathcal{O}_\kappa$ is $k$-balanced, then $\kappa$ is called a \textit{$k$-balanced coloring}.
\end{defn}

\begin{figure}[h]
\centering
\includegraphics[height=.2\textheight]{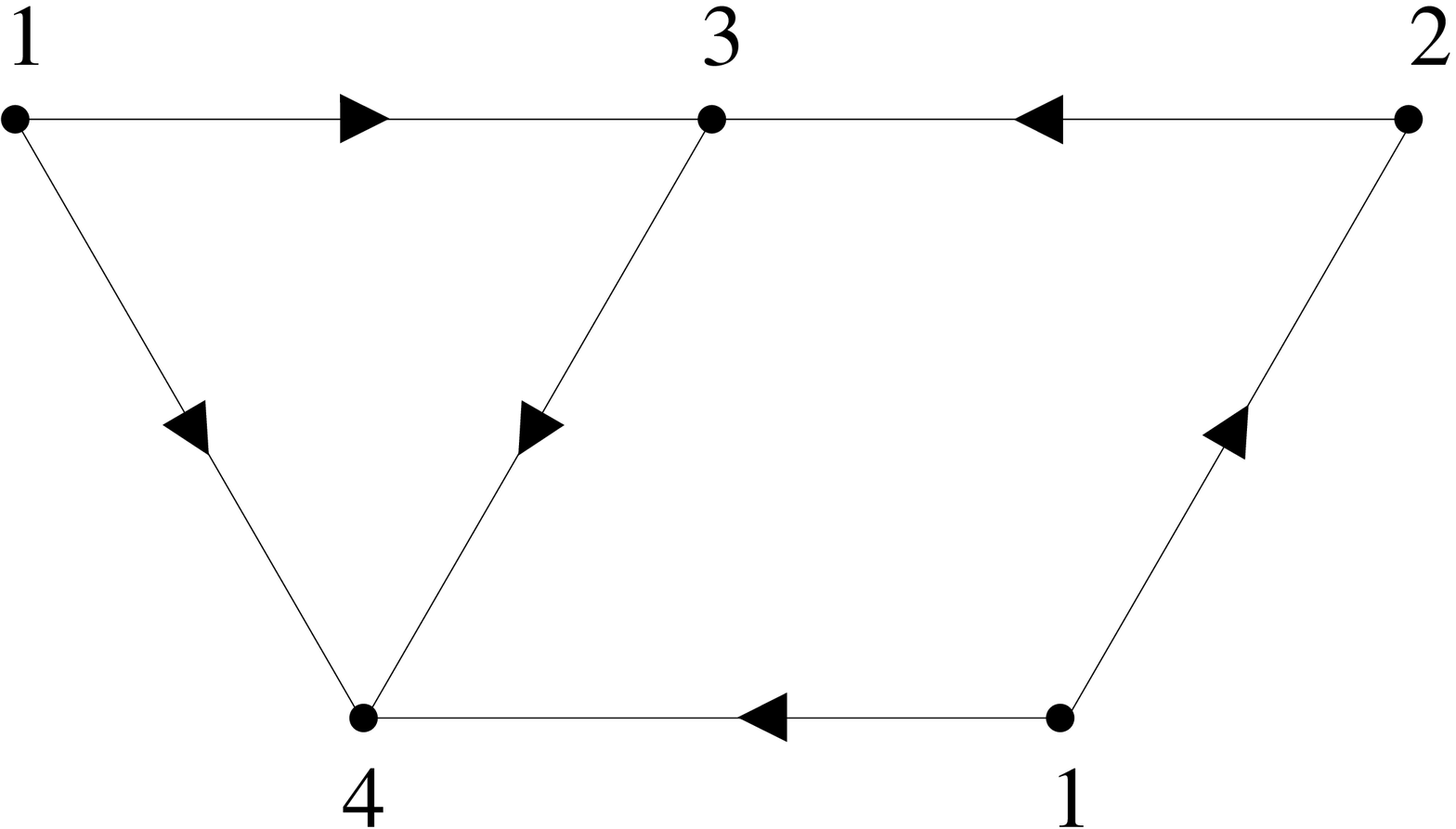}
\caption{A graph coloring and its induced orientation}
\label{ColorOrient}
\end{figure}

\begin{defn}
Given a simple graph $G$ with $n$ vertices and any positive integer $k$, define the \textit{k-balanced chromatic quasisymmetric function of G} by
$$X^k_G = \displaystyle X_G^k(x_1,x_2,\ldots) = \sum_{\kappa} x_{\kappa(1)}x_{\kappa(2)} \dots x_{\kappa(n)},$$
the sum over all $k$-balanced colorings $\kappa: V(G) \to \mathbb{P}$.
\end{defn}


To see that $X_G^k$ is indeed quasisymmetric, let $\kappa$ be a $k$-balanced coloring and let $\tau : \mathbb{N} \to \mathbb{N}$ be an order-preserving injection. Then $\kappa' = \tau \circ \kappa$ is also a proper coloring, and since $\tau$ is order-preserving, every edge of $\mathcal{O}_{\kappa'}$ is oriented identically in $\mathcal{O}_{\kappa}$ so that $\kappa'$ is also $k$-balanced. If $\tau^*$ is defined by $\tau^*(x_i) = x_{\tau(i)}$, then the previous implies that $X_G^k$ is invariant under any $\tau^*$, which is exactly the condition necessary for quasisymmetry.

In the case that $k=1$, $X_G^1$ is symmetric. In particular, a 1-balanced coloring is a proper coloring, so $X_G^1$ is Stanley's chromatic symmetric function $X_G$. In general, however, $X_G^k$ is not symmetric. The smallest counterexample is $G=K_{3,3}$ the complete bipartite graph, where $[M_{2121}]X_{K_{3,3}}^2 = 36$, but $[M_{2112}]X_{K_{3,3}}^2 = 18$.

Let $G+H$ denote the disjoint union of two graphs $G,H$. As with the chromatic symmetric function, we have that $X_{G+H}^k = X_G^k \cdot X_H^k$, which follows from the definition of $X_G^k$.

The girth $g$ of a graph $G$ plays an important role in determining $X_G^k$, as must be expected from the remarks about girth above. That is, if $k > \frac{g}{2}$, $X_G^k = 0$. As a special case, if $G$ has a triangle, then $g = 3$ and so $X_G^k = 0$ for $k > 1$. Alternately, if $g = \infty$ (that is, $G$ is a forest), then the condition that weak cycles are $k$-balanced is vacuous, so that $X_G^k = X_G$.

In Stanley's original paper on the chromatic symmetric function, he shows that $X_G$ does not distinguish between graphs generally, giving as a counterexample two graphs of girth 3 with the same chromatic symmetric function. It is also true that $X_G^k$ does not distinguish between graphs generally, for the obvious reason that $X_G^k = 0$ for all graphs with small girth. However, it is unknown whether $X_G^k$ distinguishes graphs on which the function is nonzero, nor whether $X_G^k$ distinguishes graphs that $X_G$ does not (in fact, the author is not aware of two graphs of girth greater than 3 which share the same chromatic symmetric function.) In the former case, since $X_G^k = X_G$ for $G$ a forest, if it were shown that $X_G$ does not distinguish between trees, then clearly $X_G^k$ would not distinguish between graphs.

\subsection{$L$-positivity}

We have given the $k$-balanced chromatic quasisymmetric function
using the standard monomial basis, where the coefficients count colorings.
As we now show, $X_G^k$ has a natural positive expansion in the
fundamental basis $\{L_\alpha\}$. The idea of the proof is to interpret colorings
as certain $P$-partitions.

\begin{thm} \label{Lpos}
For all graphs $G$ and for all $k$, $X_G^k$ is $L$-positive.
\end{thm}

\begin{proof}
Let $\mathcal{O}$ be any $k$-balanced orientation of $G$, and define $P_\mathcal{O}$ to be the poset induced by $\mathcal{O}$. (Notice that if $k$ = 1, $G$ may not be isomorphic to the Hasse diagram of $P_\mathcal{O}$.) 

Choose an arbitrary natural relabelling $P'_\mathcal{O}$ of $P_\mathcal{O}$. Now a $P'_\mathcal{O}$-partition is just an order-preserving map $f:P'_\mathcal{O} \to \mathbb{P}$. If we consider $f$ as a function on the undirected graph $G$, then $f$ is a coloring of $G$ which induces $\mathcal{O}$. That is to say, $f$ is a $k$-balanced coloring of $G$. Thus, any $P'_\mathcal{O}$-partition is a $k$-balanced coloring of $G$.

Conversely, any $k$-balanced coloring $\kappa$ is a $P'_{\mathcal{O}_\kappa}$-partition for the appropriate natural relabelling.

Thus,
\begin{align*}
X_G^k 	& = \sum_{\mathcal{O}} K_{P_\mathcal{O}} \\
		& = \sum_{\mathcal{O}} \sum_{\pi \in \mathcal{L}_{P'_\mathcal{O}}} L_{\co(\pi)},
\end{align*}
where the sum is over all $k$-balanced orientations $\mathcal{O}$ of $G$.
\end{proof}

\begin{figure}[h]
	\centering
	\begin{minipage}{.4\textwidth}
		\centering
		\includegraphics[width=\textwidth]{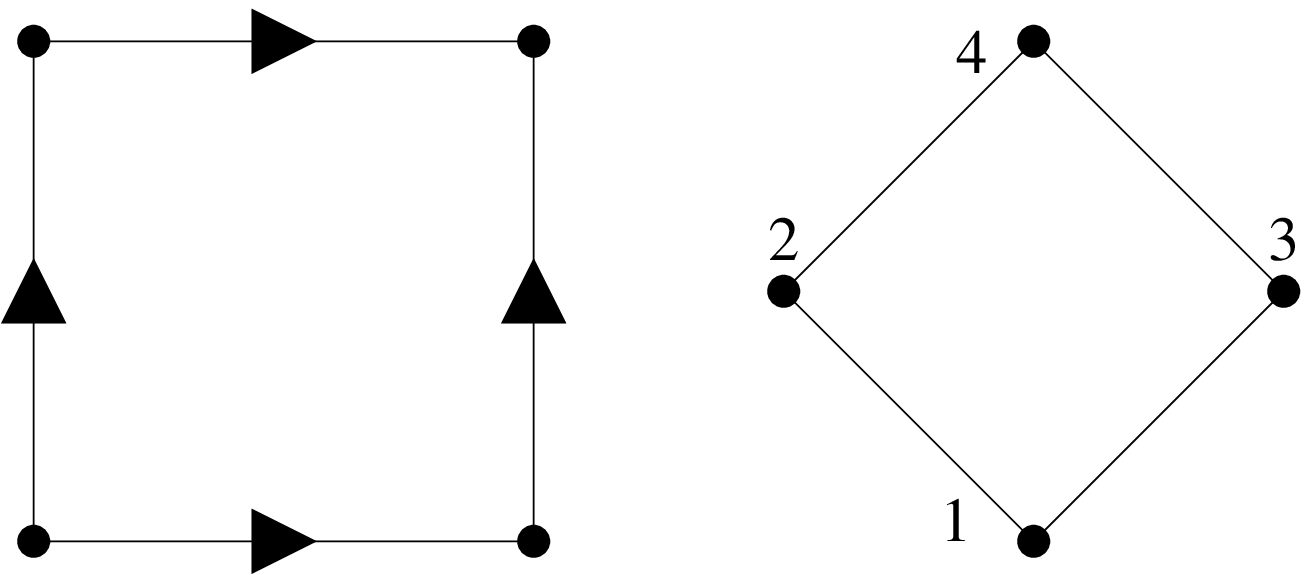}
		\caption{$\mathcal{O}_1$}
		\label{BoolOrient}
	\end{minipage}
	\hspace{.1\textwidth}
	\begin{minipage}{.4\textwidth}
		\centering
		\includegraphics[width=\textwidth]{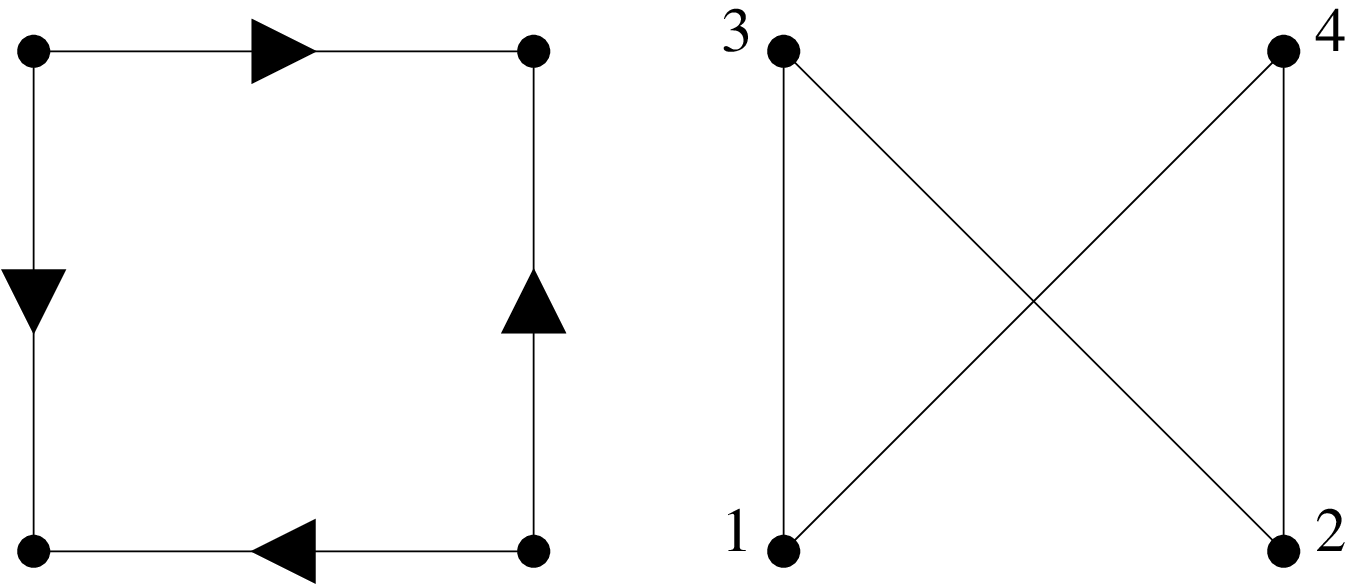}
		\caption{$\mathcal{O}_2$}
		\label{NonBoolOrient}
	\end{minipage}
\end{figure}

For an example of how this theorem works in practice, consider the 4-cycle $C_4$. The 2-balanced orientations of $C_4$ are of two types: $\mathcal{O}_1$ pictured in Figure \ref{BoolOrient}, and $\mathcal{O}_2$ pictured in Figure \ref{NonBoolOrient}. There are 4 orientations of the form $\mathcal{O}_1$ and 2 orientations of the form $\mathcal{O}_2$. We then calculate the linear extensions for $P_{\mathcal{O}_1}$ as $\{1234, 1324\}$ and for $P_{\mathcal{O}_2}$ as $\{1234, 1243, 2134, 2143\}$. Thus,
\begin{align*}
X_{C_4}^2 	& = 4(L_{1111} + L_{121}) + 2(L_{1111} + L_{112} + L_{211} + L_{22}) \\
			& = 6L_{1111} + 2L_{211} + 4L_{121} + 2L_{112} + 2L_{22}.
\end{align*}
This is in practice a much quicker way to compute $X_G^k$ than the original definition, which requires one to check the $k$-balance of every proper coloring of $G$, which in turn amounts to checking each weak cycle of the graph for each proper coloring.

\section{$X^k_G$ on special classes of graphs}

\subsection{Cycles}

Let the cycle on $n$ vertices be denoted by $C_n$. The colorings of $C_n$ which are not 2-balanced are easy to describe. Specifically, the only proper colorings which can induce a bypass (see Figure \ref{bypass}) are the colorings with $n$ distinct colors arranged in order around the cycle. We can use this to obtain the following proposition.

\begin{prop} \label{CycleCalc} For the cyclic graph $C_n$, $$X_{C_n}^2 = X_{C_n} - 2nM_{11\dots 1}.$$ \end{prop}

In particular, $X^2_{C_n}$ is symmetric for all $n$. In fact, although we do not have an explicit formula for $k\geq 3$, this fact holds for all values of $k$.

\begin{prop} \label{CycleSymm} For the cyclic graph $C_n$, $X_{C_n}^k$ is symmetric for all $k$. \end{prop}

\begin{proof}
Similar to the proof that the Schur functions are smmetric\cite[\S 7.10]{EC2}, we will show that $X_{C_n}^k$ is  invariant under changing $x_i$ to $x_{i+1}$. If $\alpha = (\alpha_1, \dots, \alpha_i, \alpha_{i+1}, \dots, \alpha_{\ell})$, let $\widetilde{\alpha} = (\alpha_1, \dots, \alpha_{i+1}, \alpha_i, \dots, \alpha_{\ell})$. Then if $\mathcal{C}_{\alpha}$ denotes the set of $k$-balanced colorings with composition type $\alpha$, we want a bijection $\varphi : \mathcal{C}_\alpha \to \mathcal{C}_{\widetilde{\alpha}}$.

Let $\kappa \in \mathcal{C}_\alpha$. The graph induced by the inverse image $\kappa^{-1}(\{i, i+1\})$ is either the entire cycle or a collection of disjoint paths. In the former case, we set $\varphi(\kappa)(v) = i$ when $\kappa(v) = i+1$ and vice versa. The preserves $k$-balance since it reverses all edges.

If $\kappa^{-1}(\{i, i+1\})$ induces a collection of paths, let $\varphi(\kappa)(v)$ swap $i$ and $i+1$ if $v$ is in such a path of odd length and otherwise set $\varphi(\kappa)(v) = \kappa(v)$. We claim that the orientation induced by $\varphi(\kappa)$ is $k$-balanced. Firstly, if $j \neq i, i+1$, then $j> i$ if and only if $j > i+1$. Thus, no edges outside of the odd lenth paths will be reoriented. Secondly, there are an even number of edges in each odd length path, with exactly half pointing each direction. The effect of $\varphi$ is to reverse all of these edges, which does not affect $k$-balance.

In either case, $\varphi$ is an involution between $\mathcal{C}_\alpha$ and $\mathcal{C}_{\widetilde{\alpha}}$, so we have the desired bijection.

\end{proof}

\subsection{Complete bipartite graphs}

For a general simple graph $G$, the coefficients of $X_G^k$ in the monomial basis directly count $k$-balanced colorings of $G$. However, in the case where $G$ is the complete bipartite graph $K_{m,n}$ and $k = 2$, there is more direct description of the coefficient any $M_\alpha$.

\begin{defn}
Let $i_1, \dots, i_k$ be positive integers. \textit{The complete ranked poset} $Q_{i_1, i_2, \ldots, i_k}$ is the poset on $\bigcup_{j=1}^k R_j$, where $|R_j| = i_j$ and each element in $R_j$ is covered by each element in $R_{j+1}$.
\end{defn}

\begin{thm} \label{CompBipartCalc}
For the complete bipartite graph $K_{m,n}$, we have
$$X_{K_{m,n}}^2 = \sum_{\alpha \in \comp(m+n)}\frac{m!n!}{\alpha!} r(\alpha; m, n)M_\alpha$$
where
$$r(\alpha;m,n) = |\{ (i,j) | 1 < i \leq j \leq \ell(\alpha) \text{ and } \sum_{t=i}^j \alpha_t = m \text{ or } n\}|$$
and
$$\alpha! = \alpha_1!\alpha_2! \cdots \alpha_{\ell(\alpha)}!.$$
\end{thm}

\begin{proof}
A 2-balanced orientation of a graph is precisely a realization of that graph as a Hasse diagram. So, we consider the posets which have Hasse diagram isomorphic to $K_{m,n}$. No such poset can have a chain of length 3, since in any chain of length 3 there must be an edge from the greatest to the smallest element, violating the fact that it is a Hassee diagram. Further, it is not hard to see that any of the complete ranked posets $Q_{i, m, n-i}$ for $0 < i \leq n$ or $Q_{i, n, m-i}$ for $0 < i \leq m$ have $K_{m,n}$ as their underlying graph. Thus, every 2-balanced orientation of $K_{m,n}$ comes from one of these posets.

We associate the coloring $\kappa$ with a composition $\alpha$, where $\alpha_i$ is the number of vertices colored with the $i^{\text{th}}$ smallest color. If a coloring agrees with one of the orientations as a complete ranked poset, no vertices of different ranks may have the same color. So, a coloring will be feasible if and only if its associated composition can be written as $\alpha = (\alpha', \alpha'', \alpha''')$, where $\alpha', \alpha'', \alpha'''$ are compositions of magnitudes either $i, m, n-i$ or $i, n, m-i$. That is, $\alpha$ comes from a feasible coloring if and only if there is a partial sum $\alpha_i + \dots + \alpha_j$ which equals $m$ or $n$. So $r(\alpha;m,n)$ counts the number of feasible colorings associated with $\alpha$ up to the number of vertices of each color.

In the case of $Q_{i,m,n-i}$, the bottom rank can be colored in $\binom{i}{\alpha'}$ ways, the middle rank colored in $\binom{m}{\alpha''}$ ways, and the top rank colored in $\binom{n-i}{\alpha'''}$ ways. Further, we must choose $i$ elements from the partite set with $n$ elements to lie in the bottom rank. Thus, the number of colorings on the poset $Q_{i, m, n-i}$ with composition type $\alpha$ is $$\binom{n}{i}\frac{i!}{\alpha'!}\frac{m!}{\alpha''!}\frac{(n-i)!}{\alpha'''!} = \frac{m!n!}{\alpha!}.$$ A similar calculation on $Q_{i,n,m-i}$ gives the same result, so that the number of 2-balanced colorings of $K_{m,n}$ with composition type $\alpha$ is $$\frac{m!n!}{\alpha!} r(\alpha; m, n)$$ as desired.

\end{proof}

\section{The $k$-balanced chromatic polynomial}

We now study the $k$-balanced versions of the chromatic polynomial of a graph.  Stanley's theorem \cite{RPS73} enumerating acyclic orientations via the chromatic polynomial turns out to have a natural generalization to the $k$-balanced setting.

\begin{defn}
The \textit{$k$-balanced chromatic polynomial of $G$} is the function $\chi_G^k : \mathbb{N} \to \mathbb{N}$ where $\chi_G^k(\lambda)$ is the number of $k$-balanced colorings of $G$ with $\lambda$ colors.
\end{defn}

Note that $\chi_G^k$ is a specialization of $X_G^k$. That is, 
$$\chi_G^k(\lambda) = X_G^k(\underbrace{1, 1, \ldots, 1}_{\lambda \text{ 1's}}, 0, 0, \ldots).$$
This allows us to prove the following fact.

\begin{prop}
The $k$-balanced chromatic polynomial $\chi_G^k(\lambda)$ is a polynomial in $\lambda$ with rational coefficients.
\end{prop}
\begin{proof}
We observe that 
\begin{align*}
\chi_G^k(\lambda) 	& = X_G^k(\underbrace{1, 1, \ldots, 1}_{\lambda \text{ 1's}}, 0, 0, \ldots) \\
				& = \sum_{i = 1}^n \sum_{\ell(\alpha) = i} c_\alpha M_\alpha(\underbrace{1, 1, \ldots, 1}_{\lambda \text{ 1's}}, 0, 0, \ldots) \\
				& = \sum_{i=1}^n \sum_{\ell(\alpha) = i} c_\alpha \binom{\lambda}{i},
\end{align*}
where the $c_\alpha$ are the integer coefficients of $M_\alpha$. In particular, $\chi_G^k$ is a polynomial in $\lambda$ with rational coefficients.
\end{proof}

It is well-known that the chromatic polynomial $\chi_G^1$ has integer coefficients. However, this is essentially the only $k$ for which this is true.

\begin{thm}
For $k > 1$, $\chi_G^k$ has integer coefficients if and only if $G$ is a forest or $G$ has no $k$-balanced coloring.
\end{thm}

\begin{proof}
If $G$ is a forest, then since $G$ has no weak cycles, any orientation is $k$-balanced for all $k$ . Thus, $\chi_G^k = \chi_G^1$ has integer coefficients. Alternately, if $G$ has no $k$-balanced coloring, then $\chi_G^k = 0$.

On the other hand, if $G$ has a cycle and a $k$-balanced coloring, then the leading coefficient of $\chi_g^k(\lambda)$ is
\begin{align*}
[\lambda^n]\chi_G^k(\lambda) & = [\lambda^n] \sum_{i=1}^n \sum_{\ell(\alpha) = i} c_\alpha \binom{\lambda}{i} \\
& = \frac{c_{11\dots 1}}{n!}.
\end{align*}

Notice that $c_{11\dots 1}$ is precisely the number of $k$-balanced colorings of $G$ with distinct colors. Since $G$ contains a cycle, there exist colorings of $G$ with distinct colors that are not $k$-balanced. That is, if the cycle consists of the vertices $v_1, v_2, \dots, v_t$ in order, assign them the colors $1, 2, \dots, t$ respectively to obtain such a coloring. Thus, $c_{11 \dots 1} < n!$. 

Since $G$ possesses a $k$-balanced coloring, it possesses a $k$-balanced coloring with distinct colors--a natural relabelling of the induced orientation will give such a coloring--so that $c_{11 \dots 1} > 0$. Thus, the leading coefficient of $\chi_G^k$ is not an integer.
\end{proof}

Stanley \cite{RPS73} proved the $k=1$ case of the following theorem --- that is, when the orientations in question are acyclic and the colorings are simply proper colorings.

\begin{thm} \label{ChrPolyEval}
$(-1)^n\chi_G^k(-\lambda)$ is the number of pairs $(\kappa, \mathcal{O})$ where

\begin{itemize}
\item $\mathcal{O}$ is a $k$-balanced orientation of $G$;
\item $\kappa$ is a proper coloring of $G$ with $\lambda$ colors;
\item $\kappa(i) \leq \kappa(j)$ if and only if $(i,j)$ is an edge of $\mathcal{O}$.
\end{itemize}
\end{thm}

\begin{proof}
Let $\Omega(P, \lambda)$ be the \textit{order polynomial} of a poset $P$ --- the number of order-preserving maps from $P$ to $[\lambda]$ --- and let $\overline {\Omega}(P, \lambda)$ be the \textit{strict order polynomial} of $P$ --- the number of strict order-preserving maps from $P$ to $[\lambda]$.

Since such a strict-order preserving map is a $P$-partition, we find that
$$K_{P_\mathcal{O}}(\underbrace{1, 1, \ldots, 1}_{\lambda \text{ 1's}}, 0, 0, \ldots) = \overline{\Omega}(P_\mathcal{O}, \lambda).$$
Applying these ideas to posets induced by orientations of $G$, we have $X_G^k = \sum_{\mathcal{O}}K_{P_\mathcal{O}}$, so that
$$\chi_G^k(\lambda) = \sum_{\mathcal{O}} \overline{\Omega}(P_\mathcal{O}, \lambda),$$
where the sum is over all $k$-balanced orientations $\mathcal{O}$.

Now we can use the fact from \cite[p. 174]{RPS73} that $\overline{\Omega}(P, -\lambda) = (-1)^{|P|}\Omega(P, \lambda)$ to get
$$\chi_G^k(-\lambda) = (-1)^n \sum_{\mathcal{O}} \Omega(P_\mathcal{O}, \lambda).$$

Lastly, an order-preserving map from $P_\mathcal{O}$ to $[\lambda]$ can be regarded as a coloring of $G$ which agrees with the orientation $\mathcal{O}$ in the above sense.
\end{proof}

\begin{cor}
$(-1)^n\chi_G^k(-1)$ is the number of $k$-balanced orientations of $G$.
\end{cor}

\bibliographystyle{acm}
\bibliography{bibliography}
\end{document}